\documentclass[12pt]{amsart}

\usepackage{amsfonts}
\usepackage{amssymb, latexsym, amsthm, amsmath, verbatim}
\usepackage{mathrsfs}
\usepackage{multirow}
\usepackage{indentfirst}
\usepackage{float}
\usepackage{enumerate}

\numberwithin{equation}{section}
\theoremstyle{definition}
\newtheorem{Thm}{Theorem}
\newtheorem{Prop}[equation]{Proposition}

\newtheorem{Lem}[equation]{Lemma}

\newtheorem{Rmk}[equation]{Remark}
\newtheorem{Conj}[equation]{Conjecture}

\makeatletter
\def\imod#1{\allowbreak\mkern5mu{\operator@font mod}\,\,#1}
\makeatother

\usepackage{color}

\definecolor{blue}{rgb}{0,0,1}
\definecolor{red}{rgb}{1,0,0}
\definecolor{green}{rgb}{0,.6,.2}
\definecolor{purple}{rgb}{1,0,1}

\long\def\red#1\endred{{\color{red}#1}}
\long\def\blue#1\endblue{{\color{blue}#1}}
\long\def\purple#1\endpurple{{\color{purple}#1}}
\long\def\green#1\endgreen{{\color{green}#1}}
\allowdisplaybreaks[4]

\begin{document}

\title[]{Rankin-Cohen Brackets of Hilbert Hecke Eigenforms}
\author{Yichao Zhang and Yang Zhou$^{\star}$}
\address{School of Mathematics,
Harbin Institute of Technology,
150001 Harbin, P.R.China}
\email{yichao.zhang@hit.edu.cn}
\address{School of Mathematics and Statistics, Hunan Normal University, Changsha 410081, P. R. China}
\email{yichao.zhang@hunnu.edu.cn}
\address{School of Mathematics, Harbin Institute of Technology, Harbin 150001, P. R. China}
\email{18B912038@stu.hit.edu.cn}
\date{}
\thanks{$^{\star}$ corresponding author}


\begin{abstract}
Over any fixed totally real number field with narrow class number one,  we prove that the Rankin-Cohen bracket of two Hecke eigenforms  for the Hilbert modular group can only be a Hecke eigenform for dimension reasons, except for a couple of cases where the Rankin-Selberg method does not apply. We shall also prove a conjecture of Freitag on the volume of Hilbert modular groups, and assuming a conjecture of Freitag on the dimension of the cuspform space, we obtain a finiteness result on eigenform product identities.
\end{abstract}
\maketitle
\newcommand{\Z}{{\mathbb Z}} 
\newcommand{\Q}{{\mathbb Q}} 
\newcommand{\R}{{\mathbb R}}
\newcommand{\C}{{\mathbb C}}
\newcommand{\N}{{\mathbb N}}
\newcommand{\h}{{\mathbb H}}
\newcommand{\SR}{{\mathrm{SL}_2(\R)}}

\section{Introduction and Statement of the Main Theorem}

Let $F$ be any fixed totally real number field of degree $n$ with narrow class number one, and $\mathfrak{d}$, $\mathcal{O}$ and $\mathcal{O}^{\times+}$ be its different, the ring of integers and  the group of totally positive units respectively. Consider the Hilbert modular group of full level
$$
\Gamma=\Gamma_{0}(\mathcal{O},\mathcal{O})=\left\{\gamma=\left(\begin{array}{ll}
a & b \\
c & d
\end{array}\right) \in\left(\begin{array}{cc}
\mathcal{O} & \mathfrak{d}^{-1} \\
 \mathfrak{d} & \mathcal{O}
\end{array}\right): \mathrm{det}(\gamma) \in \mathcal{O}^{\times+}\right\},
$$ and denote by $M_{k}\left(\Gamma\right)$ and $S_{k}\left(\Gamma\right)$ the space of Hilbert modular forms  and that of cuspforms of weight $k=(k_1,\cdots,k_n)\in \mathbb{N}^n$ for $\Gamma$ respectively.  
Here $\mathbb{N}$ is the set of natural numbers and $0\in\mathbb N$. We introduce $\mathfrak d$ in the definition of $\Gamma$  to make the Fourier expansion of Hilbert modular forms over $\mathcal{O}$.

For $m\in\mathbb{N}$, the $m$-th Rankin-Cohen bracket of $f_{1} \in M_{k}\left(\Gamma\right)$ and $f_{2} \in M_{l}\left(\Gamma\right)$ is defined by
\begin{equation*}
 \left[f_{1}, f_{2}\right]_{m}=\sum_{\substack{t \in \mathbb{N}^{n} \\
0\leq t_{i}\leq m}}(-1)^{|t|}\binom{k+m\mathbf{1}-\mathbf{1}}{m\mathbf{1}-t}\binom{
l+m\mathbf{1}-\mathbf{1}}{t}
 f_{1}^{(t)}(z) f_{2}^{(m\mathbf{1}-t)}(z),   
\end{equation*}
where $|t|=\sum_{i=1}^{n}t_i$, $\mathbf{1}=(1,\cdots,1)\in\mathbb{N}^n$, $z=(z_1,\cdots, z_n)$, $f_{i}^{(m\mathbf{1})}(z)=\frac{\partial^{nm}}{\partial z_{1}^{m} \cdots \partial z_{n}^{m}} f_{i}(z)$ for $i=1,2$ and $$\binom{k+m\mathbf{1}-\mathbf{1}}{m\mathbf{1}-t}=\prod_{j=1}^{n}\binom{k_j+m-1}{m-t_j}.$$
Clearly if $m=0$, $[f_1,f_2]_0=f_1f_2$.

When is the Rankin-Cohen bracket of two Hecke eigenforms still a Hecke eigenform? In the elliptic case and $m=0$, Duke \cite{Duke1997product} and Ghate  \cite{Ghate2000Eisenstein} proved independently that there are exactly $16$ eigenform product identities $f=gh$ for $\mathrm{SL}_2(\mathbb{Z})$, all of which hold trivially for dimension reasons. For example, since $\dim M_8(\mathrm{SL}_2(\mathbb{Z}))=\dim S_{16}(\mathrm{SL}_2(\mathbb{Z}))=1$, $E_4^2$ is a multiple of $E_8$ and $E_4\Delta_{12}$ is a multiple of $\Delta_{16}$. Here $E_k$ is the Eisenstein series of weight $k$ and $\Delta_{k}$ is the normalized cuspidal Hecke eigenform of weight $k$. Later, Johnson \cite{Johnson2013eigenforms} showed that there are $61$ eigenform identities over all levels, all weights and Nebentypes, some of which hold non-trivially. When $m>0$, Lanphier and Takloo-Bighash \cite{Lanphier2004RC} proved that there are finitely many $f$, $g$, and $m$ such that $[f,g]_m$ is an eigenform for $\mathrm{SL}_2(\mathbb{Z})$, and then Meher \cite{Meher2012RC} proved that there are no non-trivial a.e. eigenform bracket identities $f=[g,h]_m$ for general square-free level except for finitely many cases, a generalization of the result of Ghate \cite{Ghate2002eigenforms} to the case of Rankin-Cohen brackets. 

For Hilbert modular forms, when $n=2$ and $m=0$, Joshi and the first named author \cite{Joshi-Zhang2019Hilbert} proved that amongst all real quadratic fields $F$, the equation $f=gh$ has only finitely many solutions in Hecke eigenforms of full level and weight $2\cdot\mathbf{1}$ or greater. They also conjectured that amongst all totally real number fields $F$ of degree $n$  and all non-zero integral ideals $\mathfrak{n}$, there exist only finitely many solutions to the
equation
$f=gh$, where $f$, $g$, $h$ are Hecke eigenforms of level $\mathfrak{n}$ and integral weights $2\cdot\mathbf{1}$ or greater. For general $n>1$ and $m=0$, You and the first named author \cite{You-Zhang2021Hilbert} proved that for any fixed positive integer $n$, there are finitely many Hilbert eigenform product identities over all totally real number fields of degree $n$ and all eigenforms of full level and weight $2\cdot\mathbf{1}$ or greater, confirming a special case of Conjecture 8.1 in \cite{Joshi-Zhang2019Hilbert}. 

In this paper, we consider the case $n>1$ and prove that the eigenform bracket identities hold trivially with a couple of exceptions. Even in the case $m=0$, our result is different from that of \cite{You-Zhang2021Hilbert} where the authors were only concerned with finiteness. 
We shall look at the case $m=0$ in the last section and consider the above conjecture of Joshi and the first named author, under our assumption of trivial narrow class group and a conjecture of Freitag (slightly more general than Freitag's original conjecture on page 131 of \cite{Freitag1990Hilbert}). Along the way, we shall give a proof of the volume part of Freitag's conjecture for all Hilbert modular groups, which is of independent interest.

Note that the number of Eisenstein series is one since the narrow class group is trivial and hence the number of cusps is one. Note also that the Rankin-Cohen bracket
is a cusp form if $m > 0$. Here is the main theorem.

\vspace{0.3cm}
\noindent
\textbf{Main Theorem.}
 Let $F$ be a totally real number field of degree $n$ with narrow class number one. Assume that $m\in\mathbb{N}$ and $k\geq 2$ is even. 
\begin{enumerate}
     \item Let $k\geq4$ and $l\in\mathbb{N}^n$. Suppose that $g=E_k$ is an Eisenstein series and $h \in S_{l}(\Gamma)$ is a normalized  eigenform. If $\mathrm{dim}\ S_{k\mathbf{1}+l+2m\mathbf{1}}(\Gamma)> 1$, then $[g,h]_{m}$ is not an eigenform.
     \item Let $l\geq 2$ be an even integer, not equal to $k$, and $g=E_k$, $h=E_l$ be Eisenstein series. Assume that \begin{align*}
  \mathrm{dim}\ S_{k\mathbf{1}+l\mathbf{1}+2m\mathbf{1}}(\Gamma) \geq \begin{cases}1 & \text { if } m=0 \\ 2 & \text { if } m>0\end{cases}.
\end{align*}
Then $[g,h]_{m} \in M_{k\mathbf{1}+l\mathbf{1}+2m\mathbf{1}}(\Gamma)$ is not an eigenform.
 \end{enumerate}
\vspace{0.3cm}

In other words, the Rankin-Cohen bracket is an eigenform only for dimension reason in the above cases. It is immediate that the Rankin-Cohen bracket of two cuspidal eigenforms cannot be an eigenform since its $\mathcal{O}$-th coefficient would vanish. Therefore the main theorem is separated into two parts: one of $g,h$ is cuspidal or both of them are Eisenstein series.  The basic idea of the proof is that an Euler product, in this case a Rankin-Selberg $L$-function of two eigenforms,
does not vanish at any point in the region of absolute convergence. Note that the case that $g=h=E_k$ is not included in part (2) of the main theorem, due to the fact that the point to be evaluated at does not belong to the region of absolute convergence of the Rankin-Selberg $L$-function. All these ideas go back to the original proofs of Duke \cite{Duke1997product} and Ghate \cite{Ghate2002eigenforms}. These authors also treated the case of the product of two elliptic Eisenstein series of the same weight. For instance, Ghate did this by studying the growth properties of Bernoulli numbers. In this paper, we also remark on the case $g=h=E_k$, following the lines of \cite{You-Zhang2021Hilbert} and \cite{Lanphier2004RC}.

Here is the layout of this paper. In Section \ref{section2}, we set up the notations and introduce the basics. In Section \ref{section3}, we divide the main theorem into Proposition \ref{[C,E]} and Proposition \ref{[E,E]}, and give the proofs. Finally, in Section \ref{section4}, we recall a conjecture of Freitag and give a proof of its volume part for Hilbert modular groups (Theorem \ref{Conjecture-firstformula}). Then assuming the dimension part of the conjecture, we prove a finiteness result on eigenform product identities (Theorem \ref{eigenformidentity}).

\noindent {\bf Acknowledgement}: The first author is partially supported by a grant of National Natural Science Foundation of China (no. 12271123). The authors are very grateful to the anonymous referee for the useful comments on a previous version of our manuscript.

\section{Preliminaries}\label{section2}

Let $F$ be a totally real number field of degree $n$ with narrow class number one. Denote by $\mathcal{O}$ the ring of integers of $F$, $\mathcal{O}^{\times}$ the group of units, $D$ the discriminant of $F$, and $\mathfrak{d}$ the different of $F$. Let $\{\tau_1,\cdots,\tau_n\}$ be the set of $n$ real embeddings of $F$, and for simplicity we set the first real embedding $\tau_1$ to be the inclusion $F\subset \mathbb R$. Write $a_{i}=\tau_{i}(a)$ for any $a\in F$, and call an element $a$ in $F$ totally positive, denoted by $a \gg 0$, if $a_i>0$ for each real embedding $\tau_i$. We denote by $\mathcal{O}^{\times+}$ the group of totally positive units of $F$. 

Let $\mathbb{N}$ be the set of natural numbers and we assume $0\in\mathbb N$. We denote $\mathbf{1}=(1,\cdots,1)\in\mathbb{N}^n$. For $k=\left(k_{1},\cdots,k_{n}\right) \in \N^{n}$ and $z=\left(z_{1},\cdots,z_{n}\right) \in (\C^{\times})^{n}$, where  $\C^{\times}$ is the group of invertible elements of $\C$ , we employ the multi-index notation:
$$|k|=\sum _{i=1}^{n}k_i,\quad k!=\prod _{i=1}^{n}k_i!,\quad z^{k}=\prod_{i=1}^{n} z_{i}^{k_{i}}.$$
Define the trace and norm of $z\in \mathbb{C}^n$ by $$\mathrm{Tr}(z)=\sum_{i=1}^{n}z_i ,\quad N(z)=\prod_{i=1}^{n} z_{i}.$$
 For $k, l\in \mathbb{N}^n$, denote by $k\succeq l$ if $k_i\geq l_i$ for any $1\leq i \leq n$ and put $k_0=\mathrm{max}\{k_1,\cdots,k_n\}$.
 
In this paper, we only consider the full-level Hilbert modular group $\Gamma_{0}(\mathcal{O},\mathcal{O})$ defined by
$$
\Gamma_{0}(\mathcal{O},\mathcal{O})=\left\{\gamma=\left(\begin{array}{ll}
a & b \\
c & d
\end{array}\right) \in\left(\begin{array}{cc}
\mathcal{O} & \mathfrak{d}^{-1} \\
 \mathfrak{d} & \mathcal{O}
\end{array}\right): \mathrm{det}(\gamma) \in \mathcal{O}^{\times+}\right\},
$$ 
which can be embedded into $\mathrm{GL}^+_{2}(\mathbb{R})^n$ by $$\left(\begin{matrix}
a & b \\
c & d
\end{matrix}\right) \mapsto\left(\left(\begin{matrix}
a_1 & b_1 \\
c_1 & d_1
\end{matrix}\right), \cdots,\left(\begin{matrix}
a_n & b_n \\
c_n & d_n
\end{matrix}\right)\right).$$
Denote $\Gamma=\Gamma_0(\mathcal{O},\mathcal{O})$ from now on. Let $\mathbb{H}$ be the complex upper half plane. A Hilbert modular form of weight $k\in\mathbb{N}^n$ for $\Gamma$ is a holomorphic function $f$ on $\mathbb{H}^{n}$ such that 
$$
\left(f|_{k}\gamma\right)(z)=\mathrm{det}(\gamma)^{k/2}j(\gamma,z)^{-k} f(\gamma z)=f(z), \text{ for any } \gamma\in \Gamma,
$$
where $\mathrm{det}(\gamma)=(\mathrm{det}(\gamma_1),\cdots,\mathrm{det}(\gamma_n))$ and  $$\gamma z=\left(\frac{a_{1} z_{1}+b_{1}}{c_{1} z_{1}+d_{1}}, \cdots, \frac{a_{n} z_{n}+b_{n}}{c_{n} z_{n}+d_{n}}\right),\quad j(\gamma,z)=\left(c_{1} z_{1}+d_{1},\cdots,c_{n} z_{n}+d_{n}\right).$$
We denote by $M_{k}\left(\Gamma\right)$ the space of Hilbert modular forms of weight $k$ for $\Gamma$. Every $f \in M_{k}\left(\Gamma\right)$ has a unique Fourier expansion at the cusp of the form
$$f=\sum_{\nu\in \mathcal{O}}a(\nu)\mathrm{exp}(2\pi i \mathrm{Tr}(\nu z)),$$
where $\mathrm{Tr}(\nu z)=\sum_{i=1}^{n} \nu_{i} z_{i}$. Here $a(\varepsilon\nu)=\varepsilon^{\frac{k}{2}}a(\nu)$ for $\left(\begin{smallmatrix}
    \varepsilon&0\\
    0&1
\end{smallmatrix}\right)\in\Gamma$. We call $f$ a cuspform if 
$a(0)=0$ and denote by $S_{k}\left(\Gamma\right)$ the space of cuspforms. Define the Petersson inner product to be $$\langle f,h\rangle=\int_{\Gamma\backslash\mathbb{H}^n}f(z)\overline{h(z)}y^kdv,\quad dv=\prod_{j=1}^{n}\frac{dx_jdy_j}{y_j^2},\quad z_j=x_j+iy_j.$$
With this, denote by $\mathcal{E}_k(\Gamma)$ the Eisenstein subspace, namely the orthogonal complement of $S_k(\Gamma)$ in $M_k(\Gamma)$ with respect to the Petersson inner product. Denote by $\mathrm{vol}(\Gamma\backslash\mathbb{H}^n)$ the volume of $\Gamma\backslash\mathbb{H}^n$. 

Let $f_1 \in M_{k}(\Gamma)$ and $f_2 \in M_{l}(\Gamma)$ for $k$, $l\in \mathbb{N}^{n}$.
For $m\in \mathbb{N}$, the $m$-th Rankin-Cohen bracket is defined by
\begin{equation}\label{RCdef}
 \left[f_{1}, f_{2}\right]_{m}=\sum_{\substack{t \in \mathbb{N}^{n} \\
0\cdot\mathbf{1}\preceq t\preceq m\mathbf{1}}}(-1)^{|t|}\binom{k+m\mathbf{1}-\mathbf{1}}{m\mathbf{1}-t}\binom{l+m\mathbf{1}-\mathbf{1}}{t}
f_{1}^{(t)}(z) f_{2}^{(m\mathbf{1}-t)}(z),   
\end{equation}
where $f_{i}^{(m\mathbf{1})}(z)=\frac{\partial^{nm}}{\partial z_{1}^{m}\cdots \partial z_{n}^{m}} f_{i}(z)$ for $i=1,2$ and $$\binom{k+m\mathbf{1}-\mathbf{1}}{m\mathbf{1}-t}=\prod_{j=1}^{n}\binom{k_j+m-1}{m-t_j}.$$
This is a generalization of product of modular forms. By direct computation, we have  $[f_1,f_2]_m$ is equal to $[f_2,f_1]_m$ if $nm$ is even 
 and $-[f_2,f_1]_m$ otherwise. Remark that $\left[f_{1}, f_{2}\right]_{m} \in M_{k+l+2 m\mathbf{1}}\left(\Gamma\right)$,
and if $m \neq 0$, then $\left[f_{1}, f_{2}\right]_{m} \in S_{k+l+2 m\mathbf{1}}\left(\Gamma\right)$.
See \cite{Choie2007Hilbert} for more properties of the Rankin-Cohen bracket.

For $\nu\in\mathcal{O}$, the Hilbert $\nu$-th Poincar\'e series for $\Gamma$ is defined by
$$P_{k,\nu}(z)=\sum_{\gamma\in Z(\Gamma)N(\Gamma)\backslash\Gamma}(1|_k\gamma)(z)\mathrm{exp}(2\pi i \mathrm{Tr}(\nu\gamma z)),$$
where 
$$
Z(\Gamma)=\left\{\left(\begin{array}{ll}
\varepsilon & 0 \\
0 & \varepsilon
\end{array}\right): \varepsilon \in \mathcal{O}^{\times}\right\}, \quad N(\Gamma)=\left\{\left(\begin{array}{ll}
1 & b \\
0 & 1
\end{array}\right): b \in\mathfrak{d}^{-1}\right\}.
$$
Remark that $P_{k,\nu}\in S_k(\Gamma)$ if $k>2\cdot\mathbf{1}$ and $\nu\gg0$. For $$f(z)=\sum_{\substack{\nu\in\mathcal{O}\\ \nu\gg 0}}a(\nu)\mathrm{exp}(2\pi i \mathrm{Tr}(\nu z))\in S_k(\Gamma),$$
we have 
\begin{equation}\label{Poincare-coeff}
\langle f, P_{k,\nu}\rangle=a(\nu)D^{-\frac{1}{2}}(4\pi)^{n-|k|} \nu^{\mathbf{1}-k} (k-2\cdot\mathbf{1})!. 
\end{equation}
which is exactly  equation (28) in \cite{Choie2007Hilbert}.

For even $k>2$, the Eisenstein series is defined as
$$
    E_{k}(z)=\sum_{\gamma \in \Gamma_{\infty} \backslash \Gamma}\left(1|_{k\mathbf{1}} \gamma\right)(z)\in \mathcal{E}_{k\mathbf{1}}(\Gamma), 
$$
where $\Gamma_{\infty}=\left\{\left(\begin{smallmatrix}
    *&*\\
    0&*
\end{smallmatrix}\right)\in \Gamma\right\}$. For $k=2$, define 
$$E_2(z)=\zeta^{-1}_F(2)\lim_{\sigma\rightarrow 0^+}\sum_{\substack{(\alpha, \beta)\in\mathcal{O}\times\mathcal{O}\\(\alpha, \beta)\neq(0,0)}}N(\alpha z+\beta)^{-2}|N(\alpha z+\beta)|^{-\sigma},$$
where $\zeta_F(k)$ is the Dedekind zeta function of $F$. There is
a great difference between the cases $n = 1$ and $n\geq 2$. When $n\geq2$,  $E_2(z)$ is a holomorphic Hilbert modular form. But when $n=1$, $E_2(z)$ is not a holomorphic function. 
Since $F$ has narrow class number one, the dimension of $\mathcal{E}_{k\mathbf{1}}(\Gamma)$ is equal to $1$, so $E_k$ generates $\mathcal{E}_{k\mathbf{1}}(\Gamma)$.   

For any non-zero integral ideal $\mathfrak{n}=(\nu)$ with $\nu\in\mathcal{O}$, the $\mathfrak n$-th Fourier coefficient of $f$ is defined to be $c(\mathfrak{n}, f)=a(\nu)\prod_{i=1}^{n}\nu_i^{\frac{k_0-k_i}{2}}$,
where $a(\nu)$ is the $\nu$-th Fourier coefficient of $f$. For the definition of $c(\mathfrak{n}, f)$, we refer to (2.24) of \cite{Shimura1978Hilbert}. Now for each $\mathfrak n$, there is an operator $T_\mathfrak{n}$, the $\mathfrak n$-th Hecke operator, on the space of Hilbert modular forms, and $T_\mathfrak{n}$ preserves the cuspidal subspace and the Eisenstein subspace. A Hecke eigenform is a non-zero common eigenfunction for all Hecke operators $T_{\mathfrak{n}}$. The fact that the Hecke operators $T_\mathfrak{n}$ on $S_k(\Gamma)$ are self-adjoint and mutually commute implies the existence of a basis of eigenforms. The Eisenstein series $E_k$ is also an eigenform. Any Hecke eigenform $f$ must have $c(\mathcal{O},f)\neq 0$, and $f$ is called normalized if $c(\mathcal{O},f)=1$ (see page 650 of \cite{Shimura1978Hilbert}).
Note that if $f$ is a normalized eigenform, then its $T_\mathfrak{n}$-eigenvalue is nothing but $c(\mathfrak{n},f)$, and the self-adjointness of $T_\mathfrak{n}$ forces $c(\mathfrak n,f)$ to be real.

Define the $L$-function for $f\in M_k(\Gamma)$ as
$$L(s,f)=\sum_{\mathfrak{n}}c(\mathfrak{n},f)N(\mathfrak{n})^{-s},\quad s\in\mathbb{C},$$
where $\mathfrak{n}$ runs over all non-zero integral ideals of $\mathcal{O}$. This function is absolutely convergent for $\mathrm{Re}(s)>1$. In particular, when  $c(\mathfrak{n},f)=1$ for all $\mathfrak{n}$, we have the Dedekind zeta function of $F$   $$\zeta_F(s)=\sum_{\mathfrak{n}}N(\mathfrak{n})^{-s}.$$ The Rankin-Selberg $L$-function of $f\in M_k(\Gamma)$ and $h\in M_l(\Gamma)$ is defined to be
$$
L(s,f,h)=\sum_{\mathfrak{n}}c(\mathfrak{n},f)c(\mathfrak{n},h) N(\mathfrak{n})^{-s}.
$$
If $f\in S_k(\Gamma)$ is a normalized Hecke eigenform, then by the Ramanujan conjecture proved in Theorem 1 of \cite{Blasius2006Ramanujan}, for any prime ideal $\mathfrak{p}$,
$$|c(\mathfrak p, f)|\leq
2N(\mathfrak p)^{\frac{k_0-1}{2}}.$$
Therefore, if $f$ is a cuspidal eigenform and $h$ is an eigenform, then $L(s,f,h)$ is absolutely convergent for $\mathrm{Re}(s)>\frac{k_0+l_0}{2}$ or $\frac{k_0-1}{2}+l_0$ depending on whether $h$ is cuspidal or not, where $l_0=\mathrm{max}\{l_1,\cdots,l_n\}$.

\section{Proof of the Main Theorem}\label{section3}

We first consider the Rankin-Cohen bracket of a cuspidal eigenform with an Eisenstein series. Since the Eisenstein subspace is non-trivial only if the weight is even and parallel, we assume $k\geq2$ is even in this case. Recall that the Eisenstein subspace $\mathcal{E}_{k\mathbf{1}}(\Gamma)$ is one-dimensional and generated by $E_k$, so we can fix the Eisenstein series as $E_k$. 

Following the lines of \cite{Choie2007Hilbert}, we prove a couple of lemmas in our setting. First comes an explicit expression of the Rankin-Cohen bracket of $E_k$ with a Hilbert modular form. 

\begin{Lem}\label{RC-related-Poincare}
Set $l\in \mathbb{N}^n$, $m\in\mathbb{N}$ and $k\geq2$.
Let $h(z)=\sum_{\nu\in\mathcal{O}}b(\nu)\mathrm{exp}(2\pi i \mathrm{Tr}(\nu z))\in M_l(\Gamma)$. We have
$$[E_k,h]_m=\binom{k\mathbf{1}+m\mathbf{1}-\mathbf{1}}{m\mathbf{1}}\sum_{\nu\in\mathcal{O}/\mathcal{O}^{\times+}}(2\pi i \nu)^{m\mathbf{1}} b(\nu)P_{k\mathbf{1}+l+2m\mathbf{1},\nu}(z).$$
\end{Lem}
\begin{proof}
By \eqref{RCdef}, the formulas
$$\left.f^{(m\mathbf{1})}\right|_{l+2 m\mathbf{1}}M=\sum_{\substack{t\in\mathbb{N}^n\\0\cdot\mathbf{1} \preceq t \preceq m\mathbf{1}}}\binom{m\mathbf{1}}{t}\left(\frac{c}{c z+d}\right)^{m\mathbf{1}-t} \frac{(l+m\mathbf{1}-\mathbf{1})!}{(l+t-\mathbf{1}) !}\left(\left.f\right|_{l}M\right)^{(t)}$$
for any holomorphic function $f$ on $\mathbb{H}^n$
and 
$$E_k^{(t)}(z)=\sum_{M=\left(\begin{smallmatrix}
* &* \\
c &d
\end{smallmatrix}\right) \in \Gamma_{\infty} \backslash \Gamma}\mathrm{det}(M)^{-\frac{t}{2}}\left((-c)^{t} \frac{(k\mathbf{1}+t-\mathbf{1})!}{(k\mathbf{1}-\mathbf{1}) !}\right)\left(\left.1\right|_{k\mathbf{1}+t} M\right)(z),\quad t\in\mathbb{N}^n,$$ 
we have 
\begin{align*}
&\frac{(k\mathbf{1}-\mathbf{1})!m\mathbf{1}!}{(k\mathbf{1}+m\mathbf{1}-\mathbf{1}) !}\left[E_k, h\right]_m=\sum_{M\in \Gamma_{\infty}\backslash \Gamma}\left(1|_{k\mathbf{1}}M\right)(z)\left(h^{(m\mathbf{1})}|_{l+2m\mathbf{1}} M\right)(z).
\end{align*}
By the unfolding trick, we obtain that the right-hand side of the last equation is equal to 
\begin{align*}
&\sum_{M\in \Gamma_{\infty} \backslash \Gamma} \sum_{\substack{\nu\in\mathcal{O}}} \mathrm{det}(M)^{\frac{k\mathbf{1}+l+2m\mathbf{1}}{2}}j(M,z)^{-(k\mathbf{1}+l+2m\mathbf{1})}(2\pi i \nu)^{m\mathbf{1}} b(\nu) \mathrm{exp}(2 \pi i \mathrm{Tr}(\nu Mz)) \\
=&\sum_{M\in \Gamma_{\infty}\backslash \Gamma} \sum_{\varepsilon \in \mathcal{O}^{\times+}}\sum_{\substack{\nu \in\mathcal{O}/\mathcal{O}^{\times+}}} \mathrm{det}(M)^{\frac{k\mathbf{1}+l+2m\mathbf{1}}{2}}j(M,z)^{-(k\mathbf{1}+l+2m\mathbf{1})}(2 \pi i \varepsilon \nu)^{m\mathbf{1}} \\
&\qquad\times  b(\varepsilon \nu)\mathrm{exp}(2\pi i\mathrm{Tr}(\varepsilon \nu Mz)) \\
=&\sum_{M\in \Gamma_{\infty}\backslash \Gamma} \sum_{\varepsilon \in \mathcal{O}^{\times+}} \sum_{\substack{\nu\in\mathcal{O}/ \mathcal{O}^{\times+}}} \mathrm{det}\left(\left(\begin{smallmatrix}
    \varepsilon & 0\\
    0 & 1
\end{smallmatrix}\right)M\right)^{\frac{k\mathbf{1}+l+2m\mathbf{1}}{2}}j\left(\left(\begin{smallmatrix}
\varepsilon & 0\\
0 & 1    
\end{smallmatrix}\right) M, z\right)^{-(k\mathbf{1}+l+2m\mathbf{1})}(2 \pi i \nu)^{m\mathbf{1}}\\
&\qquad\times   b(\nu) \mathrm{exp}(2 \pi i \mathrm{Tr}(\nu \varepsilon Mz))\\
=&\sum_{M\in Z(\Gamma)N(\Gamma)\backslash \Gamma} \sum_{\substack{\nu\in\mathcal{O}/ \mathcal{O}^{\times+}}} \mathrm{det}(M)^{\frac{k\mathbf{1}+l+2m\mathbf{1}}{2}}j(M, z)^{-(k\mathbf{1}+l+2m\mathbf{1})}(2 \pi i \nu)^{m\mathbf{1}} b(\nu) \mathrm{exp}(2 \pi i \mathrm{Tr}(\nu Mz))\\
=&\sum_{M\in Z(\Gamma)N(\Gamma)\backslash \Gamma} \sum_{\substack{\nu\in\mathcal{O}/ \mathcal{O}^{\times+}}} \left(1|_{k\mathbf{1}+l+2m\mathbf{1}}M\right)(z)(2 \pi i \nu)^{m\mathbf{1}} b(\nu) \mathrm{exp}(2 \pi i \mathrm{Tr}(\nu Mz))\\
=&\sum_{\substack{\nu\in\mathcal{O}/ \mathcal{O}^{\times+}}}(2 \pi i \nu)^{m\mathbf{1}} b(\nu)P_{k\mathbf{1}+l+2m\mathbf{1},\nu}(z)
\end{align*}
as desired.
\end{proof}
Note that here we reprove Proposition 1 of \cite{Choie2007Hilbert} for the group $\Gamma_0(\mathcal{O},\mathcal{O})$.

Then we relate the Petersson inner product $\left\langle f,\left[E_{k}, h\right]_{m}\right\rangle$ to a special value of the Rankin-Selberg $L$-function of $f$ and $h$, which plays an important role in the proof of Proposition \ref{[C,E]} and Proposition \ref{[E,E]} below.
\begin{Lem}\label{R-S-Rankin}
Let $l\in \mathbb{N}^{n}$, $m, k\in\mathbb{N}$ and $k\geq2$. Suppose that $f \in S_{k\mathbf{1}+l+2m\mathbf{1}}\left(\Gamma\right)$ and $h \in M_{l}\left(\Gamma\right)$ with Fourier coefficients $a(\nu)$ and $b(\nu)$ for $\nu\in\mathcal{O}$ respectively.
Then
\begin{align*}
&\left\langle f,\left[E_{k}, h\right]_{m}\right\rangle \\
=&D^{-\frac{1}{2}}(\overline{2\pi i})^{nm} \frac{(k\mathbf{1}+l+2m\mathbf{1}-2\cdot\mathbf{1})!(k\mathbf{1}+m\mathbf{1}-\mathbf{1})!}{(4 \pi)^{|k\mathbf{1}+l+2m\mathbf{1}|-n}(k\mathbf{1}-\mathbf{1}) ! m\mathbf{1} !} L(k+l_0+m-1,f,h'),
\end{align*}
where $h'(z)=\sum_{\substack{\nu\in\mathcal{O}}} \overline{b(\nu)} \mathrm{exp}(2\pi i \mathrm{Tr}(\nu z))$.
\end{Lem}
\begin{proof}
By Lemma \ref{RC-related-Poincare} and  \eqref{Poincare-coeff}, we obtain that
\begin{align*}
&\langle f,[E_k,h]_m \rangle\\
=&\left\langle f,
\binom{k\mathbf{1}+m\mathbf{1}-\mathbf{1}}{m\mathbf{1}}
\sum_{\substack{\nu\in\mathcal{O}/ \mathcal{O}^{\times+} \\
\nu \gg 0}}(2\pi i \nu)^{m\mathbf{1}} b(\nu)P_{k\mathbf{1}+l+2m\mathbf{1},\nu}(z)\right\rangle\\
=&
\binom{k\mathbf{1}+m\mathbf{1}-\mathbf{1}}{m\mathbf{1}}
 D^{-\frac{1}{2}}(\overline{2 \pi i})^{nm}(k\mathbf{1}+l+2m\mathbf{1}-2\cdot\mathbf{1})!\sum_{\substack{\nu\in\mathcal{O}/\mathcal{O}^{\times+}\\ \nu\gg 0}}\frac{a(\nu)\overline{b(\nu)}\nu^{m\mathbf{1}}}{(4 \pi \nu)^{k\mathbf{1}+l+2m\mathbf{1}-\mathbf{1}}}  \\
=&D^{-\frac{1}{2}}(\overline{2\pi i})^{nm} \frac{(k\mathbf{1}+l+2m\mathbf{1}-2\cdot\mathbf{1}) !(k\mathbf{1}+m\mathbf{1}-\mathbf{1})!}{(4 \pi)^{|k\mathbf{1}+l+2m\mathbf{1}|-n}( k\mathbf{1}-\mathbf{1})!m\mathbf{1} !} \sum_{\substack{\nu\in \mathcal{O}/\mathcal{O}^{\times+} \\
\nu\gg0}} \frac{a(\nu)\overline{b(\nu)}}{\nu^{k\mathbf{1}+l+m\mathbf{1}-\mathbf{1}}}\\
=&D^{-\frac{1}{2}}(\overline{2\pi i})^{nm} \frac{(k\mathbf{1}+l+2m\mathbf{1}-2\cdot\mathbf{1}) !(k\mathbf{1}+m\mathbf{1}-\mathbf{1})!}{(4 \pi)^{|k\mathbf{1}+l+2m\mathbf{1}|-n}( k\mathbf{1}-\mathbf{1})!m\mathbf{1} !}\sum_{\substack{\nu\in\mathcal{O}/\mathcal{O}^{\times+}\\ \nu\gg 0}} \frac{a(\nu) \overline{b(\nu)}\nu^{l_0\mathbf{1}-l}}{N(\nu)^{k+l_0+m-1}}\\
=&D^{-\frac{1}{2}}(\overline{2\pi i})^{nm} \frac{(k\mathbf{1}+l+2m\mathbf{1}-2\cdot\mathbf{1}) !(k\mathbf{1}+m\mathbf{1}-\mathbf{1})!}{(4 \pi)^{|k\mathbf{1}+l+2m\mathbf{1}|-n}( k\mathbf{1}-\mathbf{1})!m\mathbf{1} !}L(k+l_0+m-1,f,h'). 
\end{align*}
We complete the proof.
\end{proof}

Note that the Rankin-Cohen brackets can be defined for general degrees $m\in\mathbb N^n$, and Lemma \ref{RC-related-Poincare} and Lemma \ref{R-S-Rankin} are still valid in such setting (see \cite{Choie2007Hilbert} for details). Since such general setting does not play a role in our treatment, we omit it.

\begin{Prop}\label{[C,E]}
Put $l\in\mathbb{N}^n$, $k,m\in\mathbb{N}$ and $k\geq4$. Suppose that $g=E_k \in \mathcal{E}_{k\mathbf{1}}(\Gamma)$ and $h\in S_l(\Gamma)$ is a normalized eigenform. If $\text{dim } S_{k\mathbf{1}+l+2m\mathbf{1}}(\Gamma)>1$, then $[g,h]_{m}$ is not an eigenform.
\end{Prop}
\begin{proof}
Suppose that $[g,h]_{m}$ is an eigenform. Denote the Fourier coefficients of $h$ by
$b(\nu)$, $\nu\in\mathcal{O}$ and we know that $b(\nu)$'s are real. Since $\mathrm{dim}\ S_{k\mathbf{1}+l+2m\mathbf{1}}(\Gamma)>1$, we can pick an eigenform $f\in S_{k\mathbf{1}+l+2m\mathbf{1}}(\Gamma)$ distinct from the eigenform $[g,h]_{m}$. Let $a(\nu)$ be the Fourier coefficient of $f$, $\nu\in\mathcal{O}$. 

By Lemma \ref{R-S-Rankin}, we have
\begin{align*}
&\left\langle f,\left[g, h\right]_{m}\right\rangle \\
=&D^{-\frac{1}{2}}(\overline{2\pi i})^{nm} \frac{(k\mathbf{1}+l+2m\mathbf{1}-2\cdot\mathbf{1})!(k\mathbf{1}+m\mathbf{1}-\mathbf{1})!}{(4 \pi)^{|k\mathbf{1}+l+2m\mathbf{1}|-n}(k\mathbf{1}-\mathbf{1})!m\mathbf{1}!} \sum_{\substack{\nu\in\mathcal{O}/\mathcal{O}^{\times+} \\
\nu\gg0}} \frac{a(\nu)b(\nu)}{\nu^{k\mathbf{1}+l+m\mathbf{1}-\mathbf{1}}}\\
=&D^{-\frac{1}{2}}(\overline{2\pi i})^{nm} \frac{(k\mathbf{1}+l+2m\mathbf{1}-2\cdot\mathbf{1})!(k\mathbf{1}+m\mathbf{1}-\mathbf{1})!}{(4 \pi)^{|k\mathbf{1}+l+2m\mathbf{1}|-n}(k\mathbf{1}-\mathbf{1})!m\mathbf{1}!}L\left(k+l_0+m-1,f,h\right).
\end{align*}
The left-hand side vanishes by the orthogonality of Hecke eigenforms, so we only have to prove that the right-hand side does not vanish to derive a contradiction. 

Indeed, since both $f$ and $h$ are cuspidal eigenforms, $L(s, f, h)$ has an Euler product which converges absolutely in the region $\mathrm{Re}(s)>\frac{k+2l_0+2m}{2}$. Since $k+l_0+m-1$ lies in this region when $k>2$, we have  $L\left(k+l_0+m-1, f, h
\right)\neq 0$. Hence the right-hand side does not vanish and we finish the proof.
\end{proof}
\begin{Rmk}
Recall that $E_2(z)$, unlike the elliptic case, is also a Hilbert modular form when $n\geq2$. The grand Riemann hypothesis \cite{Borwein2008Riemann} predicts that all zeros of a normalized automorphic $L$-function with $0<\mathrm{Re}(s)<1$ lie on the line $\mathrm{Re}(s)=\frac{1}{2}$. So all zeros of  the Rankin-Selberg $L$-function $L(s,f,h)$ for $f$ of weight $k\in\mathbb{N}^n$ and $h$ of weight $l\in\mathbb{N}^n$ are assumed to lie on the line $\frac{k_0+l_0-1}{2}$ by the functional equation introduced in \cite[Theorem 1 in Section 9.5]{Hida1993elementary}. Hence Proposition \ref{[C,E]} still holds for $k=2$ under the grand Riemann hypothesis, since $k+l_0+m-1>\frac{k+2l_0+2m-1}{2}$ when $k=2$.
\end{Rmk}

Now we consider the case that both $g$ and $h$ are Eisenstein series.

\begin{Prop}\label{[E,E]}
Let $k,l, m \in \mathbb{N}$, $k,l\geq2$ and $k\neq l$. Let $g=E_k(z)$ and $h=E_l(z)$ be two eigenforms. Assume that \begin{align*}
  \text{dim } S_{k\mathbf{1}+l\mathbf{1}+2m\mathbf{1}}(\Gamma) \geq \begin{cases}1 & \text { if } m=0 \\ 2 & \text { if } m>0\end{cases},
\end{align*}
then $[g,h]_{m} \in M_{k\mathbf{1}+l\mathbf{1}+2m\mathbf{1}}(\Gamma)$ is not an eigenform.
\end{Prop}
\begin{proof}
Suppose that $[g,h]_{m} \in M_{k\mathbf{1}+l\mathbf{1}+2m\mathbf{1}}(\Gamma)$ is an eigenform. If $m=0$, then $[g,h]_{0}=gh\in\mathcal{E}_{k\mathbf{1}+l\mathbf{1}}(\Gamma)$ since the constant term of $gh$ is not $0$ and $gh$ is an eigenform. If $m>0$, then  $[g,h]_{m}\in {S}_{k\mathbf{1}+l\mathbf{1}+2m\mathbf{1}}(\Gamma)$ (see \cite{Choie2007Hilbert} for more details). Hence we can always choose a normalized eigenform $f\in S_{k\mathbf{1}+l\mathbf{1}+2m\mathbf{1}}(\Gamma)$ distinct from $[g,h]_{m}$ under the assumption on $\text{dim }S_{k\mathbf{1}+l\mathbf{1}+2m\mathbf{1}}(\Gamma)$ such that $\left\langle f,\left[g, h\right]_{m}\right\rangle=0$. 

By Lemma \ref{R-S-Rankin},
\begin{align}\label{RS-[C,C]}
&\left\langle f,\left[g, h\right]_{m}\right\rangle \\\notag
=&D^{-\frac{1}{2}}(\overline{2 \pi i})^{nm} \frac{(k\mathbf{1}+l\mathbf{1}+2m\mathbf{1}-2\cdot\mathbf{1}) !(k\mathbf{1}+m\mathbf{1}-\mathbf{1})!}{(4 \pi)^{n(k+l+2m-1)}(k\mathbf{1}-\mathbf{1})!m\mathbf{1}!}L(k+l+m-1,f,h).
\end{align}
By the fourth equation on Page 674 of \cite{Shimura1978Hilbert}, we also have 
\begin{equation}\label{L-multiple}
L(k+l+m-1, f, h)=\frac{L(k+l+m-1,f) L(k+m,f)}{\zeta_F(k)}.
\end{equation}
It is obvious that the denominator of the right-hand side of \eqref{L-multiple} does not vanish, and we only have to show that the $L$-values in the numerator of the right-hand side does not vanish either.

Proposition 4.16 in \cite{Shimura1978Hilbert} shows that $L(s, f) \neq 0$ in the region $\mathrm{Re}(s)\geq(k+l+2m+1)/2$. If $ k+m\geq (k+l+2m+1)/2$, that is $k\geq l+1$, then the numerator on the right-hand side of \eqref{L-multiple} is non-vanishing. If $k<l+1$, we switch $g$ and $h$ and obtain again the desired non-vanishing in the case $l\geq k+1$ since $[h,g]_{m}=[g,h]_{m}$ or $-[g,h]_{m}$ depending on $nm$ is even or not.
This leaves out only the case $k=l$, which we have excluded. 
\end{proof}

\begin{Rmk}
Note that the case $k=l$ is missing in Proposition \ref{[E,E]}. In this case, following the lines in \cite{Lanphier2004RC}, we can prove that for any fixed $n$ and $m>0$,  $[E_{k},E_{k}]_m$ is no longer an eigenform if $k+m\geq k_0+4$ and $\mathrm{dim~}S_{2k\mathbf{1}+2m\mathbf{1}}(\Gamma)\geq2$, where $k_0$ is the largest even integer $k$ such that either $[E_k,E_k]_1$ or $[E_k,E_k]_2$ is an eigenform. Indeed, for even $n$, note first that such $k_0$ exists by the Hecke relations of the Fourier coefficients of $[E_{k},E_{k}]_1$ or $[E_{k},E_{k}]_2$, which can be proven by an argument similar to that of Section 3 in \cite{You-Zhang2021Hilbert}. Then for $k\geq k_0+2$, there exist two different normalized eigenforms $f,g\in S_{2k\mathbf{1}+2\cdot\mathbf{1}}(\Gamma)$ such that 
$$\langle f,[E_{k},E_{k}]_{1}\rangle\neq0, \quad \langle g,[E_{k},E_{k}]_{1}\rangle\neq0.$$
It follows that  $L(k+1,f)\neq0$ and  $L(k+1,g)\neq0$ by \eqref{RS-[C,C]} and \eqref{L-multiple}. Therefore, for general $m>1$, if $k+m-1\geq k_0+2$ and  $\mathrm{dim~}S_{2k\mathbf{1}+2m\mathbf{1}}(\Gamma)\geq2$, we also have the existence of two distinct eigenforms $f,g\in S_{2k\mathbf{1}+2m\mathbf{1}}(\Gamma)$ such that $L(k + m, f)\neq 0$ and $L(k + m, g)\neq 0$, implying that $[E_k, E_k]_m$ is not an eigenform by \eqref{RS-[C,C]} and \eqref{L-multiple}. When $n$ is odd, note that $[E_k, E_k]_1=0$ and this case follows similarly by simply replacing $[E_k, E_k]_1$ with $[E_k, E_k]_2$ and  $k+m-1\geq k_0+2$ with $k+m-2\geq k_0+2$ 
in the above argument.
\end{Rmk}

\section{A Conjecture of Freitag and Eigenform Product Identities}\label{section4}

In this section, we go back to the case $m=0$ and consider eigenform product identities.

By the main theorem, if there are only finitely many totally real number fields and finitely many weights $k,l$ such that $\text{dim }S_{k+l}(\Gamma)\leq1$, then we can vary $n$ in the main theorem of \cite{You-Zhang2021Hilbert}, and obtain that there are finitely many Rankin-Cohen
product eigenform identities, with a couple of exceptions and still under the extra assumption of narrow class number being equal to one.

This leads naturally to the problem of finding the asymptotic formula of $\text{dim }S_{k}(\Gamma)$. We conjecture the following, and will give a proof of the volume part for Hilbert modular groups.

Recall that $|k|=\sum_{i=1}^{n}k_i$ and  $N$ is the norm  defined before.
\begin{Conj}\label{Conjecture}
Let $\Gamma_{m}\subset\mathrm{GL}^+_2(\mathbb{R})^n$ be a sequence of groups commensurable with a Hilbert modular group such that for distinct $m,m'$ with $n=n(m)=n(m')$ the groups $\Gamma_m$ and $\Gamma_{m'}$ are not conjugate in $\mathrm{GL}^+_2(\mathbb{R})^n$. Then
\begin{align*}
&\lim_{m\rightarrow\infty}(4\pi)^{-n}\mathrm{vol}(\Gamma_{m}\backslash\mathbb{H}^n)=\infty, \\
\lim_{\substack{|k|+m\rightarrow \infty}}&\frac{\text{dim }S_k(\Gamma_{m})-(4\pi)^{-n}N(k-\mathbf{1})\mathrm{vol}(\Gamma_{m}\backslash\mathbb{H}^n)}{(4\pi)^{-n}N(k-\mathbf{1})\mathrm{vol}(\Gamma_{m}\backslash\mathbb{H}^n)}=0.
\end{align*}
\end{Conj}
Conjecture \ref{Conjecture} is slightly stronger than the original conjecture introduced by Freitag on Page 131 of \cite{Freitag1990Hilbert}. The extra scalar $(4\pi)^{-n}$ results from different normalizations of the volume. Note that the sequence $\Gamma_{m}\subset\mathrm{GL}^+_2(\mathbb{R})^n$ and we include general weights. By Theorem $4.8$ on Page $122$ of \cite{Freitag1990Hilbert}, the numerator of the second formula of Freitag's conjecture is exactly the contribution of elliptic points and cusps to the dimension of cuspform space of weight $2\cdot\mathbf{1}$, and the responding denominator is the contribution of $\mathrm{vol}(\Gamma_m\backslash \mathbb{H}^n)$ to the dimension. It is natural to conjecture that the same statement holds for general weight, with the extra factor $N(k-\mathbf{1})$ added according to the dimension formula on general weight introduced in Theorem $3.5$ on Page $110$ of \cite{Freitag1990Hilbert}. The dimension part of the conjecture is reasonable since the contributions of elliptic points and cusps to $\text{dim }S_k(\Gamma_{m})$ are expected to be small. Freitag's conjecture has been proved for the case $n=1$ in \cite{Thompson1979finiteness} and the case of Hilbert modular groups when $n=2$ in \cite{Freitag1990Hilbert}, but remains open for other cases. 

Denote by $\Gamma_{F}$ the Hilbert modular group $\Gamma$ for $F$ and $\mathcal{O}^{\times+}_{F}$ the group of totally positive units of $F$. We first prove that different totally real number fields have non-conjugate Hilbert modular groups.

\begin{Lem}\label{Conjugate}
If $F$ and $F'$ are different totally real number fields of degree $n$, then $\Gamma_{F}$ and $\Gamma_{F'}$ are not conjugate in $\mathrm{GL}^+_2(\mathbb{R})^n$. 
\end{Lem}
\begin{proof}
Let $\{\tau_1,\cdots,\tau_n\}$ and $\{\sigma_1,\cdots,\sigma_n\}$ be the set of real embeddings of $F,F'$ with $\tau_1,\sigma_1$ being the inclusions respectively. Suppose that $\Gamma_F$ is conjugate to $\Gamma_{F'}$ in $\mathrm{GL}^+_2(\mathbb{R})^n$, so there exists  $$A=(A_1,\cdots,A_n)=\left(\left(\begin{matrix}a_1 & b_1 \\ c_1 & d_1\end{matrix}\right),\cdots,\left(\begin{matrix}a_n & b_n \\ c_n & d_n\end{matrix}\right)\right)\in\mathrm{GL}^{+}_2(\mathbb R)^n$$
such that $A\Gamma_FA^{-1}=\Gamma_{F'}$. By rescaling, we may assume that $A\in \mathrm{SL}_2(\mathbb R)^n$. To show that $F=F'$, by symmetry we will prove that $F\subset F'$. To this end, we will focus on the first coordinate and only employ the equality $A_1\Gamma_FA_1^{-1}=\Gamma_{F'}$. Here the embeddings $\tau_1,\sigma_1$ are implicit.

We first prove that $\mathcal{O}^{\times+}_F\subset F'$. Fix any $\varepsilon\in \mathcal{O}^{\times+}_F$ and consider the matrix
\[
B_\varepsilon=\begin{pmatrix}
    \varepsilon & 0\\0&1
\end{pmatrix}\in \Gamma_F.
\]
Our assumption implies that \[A_1B_\varepsilon  A_1^{-1}=\begin{pmatrix}
a_1d_1\varepsilon-b_1c_1& a_1b_1(1-\varepsilon)\\
c_1d_1(\varepsilon-1)& a_1d_1-\varepsilon b_1c_1
\end{pmatrix}
\in\Gamma_{F'}.\] Since $a_1d_1-b_1c_1=1$, from the diagonal entries, we obtain that
\[\varepsilon-b_1c_1(1-\varepsilon),\quad 1+b_1c_1(1-\varepsilon)\quad \in F'.\]
Therefore, $1+\varepsilon$, hence $\varepsilon$ belongs to $F'$. Therefore, $\mathcal{O}^{\times+}_F\subset F'$ as desired.

We then prove that $\mathcal{O}^{\times+}_F$ generates $F$, from which the statement $F\subset F'$ follows trivially. The subfield $L$ of $F$ generated by $\mathcal{O}^{\times+}_F$ is again totally real and has its group of units containing $\mathcal{O}^{\times+}_F$. But by Dirichlet unit theorem, $\mathcal{O}^{\times+}_F$ is a free abelian group of rank $n-1$, so again by Dirichlet's unit theorem, the degree of $L$ must be at least $n$. This forces that $L=F$, completing the whole proof.
\end{proof}

Now we prove the volume part of Conjecture \ref{Conjecture} for the sequence of Hilbert modular groups.

\begin{Thm}\label{Conjecture-firstformula}
Let $\{F_m\}$ be a sequence of totally real number fields. Let $n(m), D(m)$ and $\Gamma_m$ be the degree, the discriminant and the Hilbert modular group of $F_m$ respectively. Then   
$\lim_{m\rightarrow\infty}(4\pi)^{-n}\mathrm{vol}(\Gamma_{m}\backslash\mathbb{H}^n)=\infty$.
\end{Thm}
\begin{proof}
The preceding lemma shows that the sequence of Hilbert modular groups over all totally real number fields satisfies the assumption of Conjecture \ref{Conjecture}.

It is clear that we only need to prove the statement for all totally real number fields.  Since rearranging a sequence does not affect its limit, we may enumerate $F_m$ as follows: by Hermite's theorem \cite[(2.16) in Section 3]{Neukirch1999Algebraic}, there are finitely many number fields with bounded discriminant, so the number of totally real number fields with fixed $n+D$ is finite and hence we may enumerate all totally real number fields as $n+D$ grows. 

We fix an integer $n\geq 2$ and consider the subsequence of $F_{m}$ with $n(m)=n$. By (2.31) of \cite{Shimura1978Hilbert}, the normalized volume
\begin{align*}
(4\pi)^{-n}\mathrm{vol}(\Gamma_m\backslash\mathbb{H}^n)&=2^{1-2n}\pi^{-2n}D(m)^{\frac{3}{2}}\zeta_F(2)[\mathcal{O}^{\times+}:{\mathcal{O}^{\times}}^2]^{-1}\\
&>2^{1-2n}\pi^{-2n}D(m)^{\frac{3}{2}}2^{1-n}=2^{2-3n}\pi^{-2n}D(m)^{\frac{3}{2}},
\end{align*}
since $\zeta_F(2)>1$ and $[\mathcal{O}^{\times+}:{\mathcal{O}^{\times}}^2]\leq 2^{n-1}$. It follows that for the subsequence of fixed $n$ the limit of the normalized volume is $\infty$, since now $D(m)\rightarrow\infty$.

To take care of the whole sequence, we recall Stark's bound \cite[(1.5)]{Odlyzko1990bounds} for a general totally real number field $F$ with degree $n$ and discriminant $D$, 
$$\log D\geq n(\log \pi-\psi(s/2))-\frac{2}{s}-\frac{2}{s-1},$$
with $s=1+\frac{1}{\sqrt{n}}$ and $\psi$ being the digamma function. Since $\psi(1/2)\approx -1.96351$ and $\psi$ is increasing on $(0,\infty)$, there exists a positive integer $n_0$ such that for all $n>n_0$, $\psi(s/2)< -1.9$.  Since the normalized volume of all subsequences with $n(m)\leq n_0$ approaches $\infty$, we only have to show that after removing these subsequences the normalized volume of resulting sequence approaches $\infty$.

For each field $F$ in the resulting sequence, since $n(m)>n_0$, Stark's bound implies that
\[\log D> n(\log \pi+1.9)-2(1+\sqrt{n})\]
or $D> \pi^ne^{1.9n}/e^{2(1+\sqrt{n})}$. One can verify the following inequality for concrete constants
\[(\pi e^{1.9})^\frac{3}{2}>8\pi^2.\]
Therefore we can choose an absolute constant $\delta>0$, small enough, such that \[(\pi e^{1.9})^{\frac{3}{2}-\delta}>8\pi^2.\]
Set $\rho=(\pi e^{1.9})^{\frac{3}{2}-\delta}/8\pi^2$, so $\rho>1$ and we have
\begin{align*}
(4\pi)^{-n}\mathrm{vol}(\Gamma_m\backslash\mathbb{H}^n)&=2^{1-2n}\pi^{-2n}D^{\frac{3}{2}}\zeta_F(2)[\mathcal{O}^{\times+}:{\mathcal{O}^{\times}}^2]^{-1}\\
&>2^{2-3n}\pi^{-2n}D^\delta D^{\frac{3}{2}-\delta}\\
&\geq 2^{2-3n}\pi^{-2n}D^\delta (\pi^ne^{1.9n}e^{-2(1+\sqrt{n})})^{\frac{3}{2}-\delta}\\
&\geq 2^2D^\delta (\rho^n e^{-3(1+\sqrt{n})})\\
&=2^2D^\delta e^{n\log\rho -3(1+\sqrt{n})},
\end{align*}
where for ease of notation we wrote $n,D$ for $n(m), D(m)$ respectively. Now it is clear that if $m$ is big, then $n$ or $D$ is big, so the normalized volume approaches $\infty$ since $\log \rho>0$. This completes the proof.
\end{proof}

With the asymptotic formula for dimension in Conjecture \ref{Conjecture}, we finally derive the following finiteness theorem on eigenform product identities for Hilbert modular groups over all totally real number fields of narrow class number one. 

\begin{Thm}\label{eigenformidentity}
Assuming the dimension part of Conjecture \ref{Conjecture}, amongst all totally real number fields $F$ of narrow class number one, the equation $f=gh$ has only finitely many solutions in the triple $(f,g,h)$ with $f,g,h$ being Hecke eigenforms of full level and weight $2\cdot\mathbf{1}$ or greater, $g,h$ being normalized and $g$ an Eisenstein series such that $g\neq h$ if $h$ is an Eisenstein series and $g$ having parallel weight greater than $2$ if $h$ is cuspidal.
\end{Thm}
\begin{proof}
Suppose that $g$ and $h$ are of weight $k$ and $l$ respectively. Note that the sequence of Hilbert modular groups $\{\Gamma_m\}$ satisfies the assumption in Conjecture \ref{Conjecture} by Lemma \ref{Conjugate}. 
 
Note that if $n$ is big, so is $m$. So by Theorem \ref{Conjecture-firstformula}, $(4\pi)^{-n}\mathrm{vol}(\Gamma_{m}\backslash\mathbb{H}^n)$ tends to $\infty$ as $n$ approaches $\infty$. 
So $\text{dim }S_{k+l}(\Gamma_{m})>1$
for all $k+l\succeq2\cdot\mathbf{1}$ when $n$ is large enough, since the main term of  
$
\text{dim }S_{k+l}(\Gamma_{m})$ is   $(4\pi)^{-n}\mathrm{vol}(\Gamma_{m}\backslash\mathbb{H}^n)N(k+l-\mathbf{1})$ as $|k|+|l|+m$ grows by Conjecture \ref{Conjecture}. It follows that, when $n$ is large enough, the equation $f=gh$ has no solutions in the chosen triple $(f,g,h)$ by Proposition \ref{[C,E]} and Proposition \ref{[E,E]}. 

But You and the first named author \cite{You-Zhang2021Hilbert} proved that for any fixed $n$, there are only finitely many eigenform product identities over all totally real number fields of degree $n$ and all eigenforms of weight $2\cdot\mathbf{1}$ or greater. Hence  the equation $f=gh$ has only finitely many solutions in the chosen triple $(f,g,h)$ over all totally real number fields of narrrow class number one and all eigenforms of weight $2\cdot\mathbf{1}$ or greater, under the assumption in the statement. We complete the proof.
\end{proof}

\bibliographystyle{amsplain}

\end{document}